\documentclass[11pt]{amsart}
\usepackage{amssymb, graphicx, amsthm}
\usepackage{epsfig}
\usepackage{verbatim}
\usepackage{amsfonts}
\usepackage{mathrsfs}
\usepackage{amsbsy}
\usepackage{graphicx}
\renewcommand{\eqref}[1]{(\ref{#1})}

\newtheorem{prop}{Proposition}[section]
\newtheorem{lem}[prop]{Lemma}

\newtheorem{thm}{Theorem}[section]

\newtheorem{cor}[prop]{Corollary}

\usepackage{amscd}

\begin{document}

\title[pseudo-Anosov maps on punctured Riemann spheres]{On pseudo-Anosov maps with small dilatations on punctured Riemann spheres}
\author[C. Zhang]{C. Zhang}
\date{January 31, 2011}
\thanks{ }

\address{Department of Mathematics \\ Morehouse College
\\ Atlanta, GA 30314, USA.}
\email{czhang@morehouse.edu}


\maketitle 

\begin{abstract}
Let $S_n$ be a punctured Riemann spheres $\mathbf{S}^2\backslash \{x_1,\ldots, x_n\}$. In this paper, we investigate pseudo-Anosov maps on $S_n$ that are isotopic to the identity on $S_n\cup \{x_n\}$ and have the smallest possible dilatations. We show that those maps cannot be obtained from Thurston's construction (that is the products of two Dehn twists). We also prove that those pseudo-Anosov maps $f$ on $S_n$ with the minimum dilatations can never define a trivial mapping class as any puncture $x_i$ of $S_n$ is filled in. The main tool is to give both lower and upper bounds estimations for dilatations $\lambda(f)$ of those pseudo-Anosov maps $f$ on $S_n$ isotopic to the identity as a puncture $x_i$ of $S_n$ is filled in.
\end{abstract}

\bigskip

\section{Introduction}
\setcounter{equation}{0}

Let $S$ be an analytically finite  Riemann surface of type $(p,n)$, where $p$ is the genus and $n$ is the number of punctures on $S$. Assume that $3p+n>3$. According to Thurston \cite{Th}, an orientation preserving homeomorphism $f$ of $S$ is called pseudo-Anosov if there is a real number $\lambda>1$ and a pair of transverse measured foliations $\{\mathcal{F}_+, \mathcal{F}_-\}$ on $S$ invariant under $f$ such that 
$$
f(\mathcal{F}_+)=\lambda\ \mathcal{F}_+ \ \ \mbox{and}\ \  f(\mathcal{F}_-)=\tfrac{1}{\lambda}\ \mathcal{F}_-.
$$
The number $\lambda=\lambda(f)$ is algebraic and is called the dilatation of $f$.

Let Mod$_S$ be the mapping class group that consists of isotopy classes of orientation preserving homeomorphisms of $S$ and let $A\subset \mbox{Mod}_S$ be a subset. Let 
$$
\mbox{spec}(A)=\left\{ \log \lambda(f):  f\in A \ \mbox{is pseudo-Anosov}   \right\}.
$$ 
Let $L(A)=\mbox{inf spec}(A)$, where the infimun is taken over all pseudo-Anosov maps in $A$. By a theorem of Ivanov \cite{Iva}, there exists an element $f\in A$ such that $L(A)=\log \lambda(f)$. 

An interesting problem in mapping class groups is to estimate $L(A)$ for various subsets $A$ of $\mbox{Mod}_S$. In \cite{Pen}, Penner constructed a pseudo-Anosov map $f$ on a genus $p\geq 2$ compact Riemann surface $S$ with a small dilatation $\lambda(f)$ and showed that
\begin{equation*}
L(\mbox{Mod}_S) < \frac{\log 11}{p}.
\end{equation*}
When $S$ is a surface of type $(p,n)$ with $p>0$ and $n>0$, an upper bound for $L(\mbox{Mod}_S)$ is in general unknown. Nevertheless, Penner's result \cite{Pen} gave a lower bound for $L(\mbox{Mod}_S)$:
\begin{equation}\label{PEN1}
L(\mbox{Mod}_S)> \frac{\log 2}{12p-12+4n}.
\end{equation}
Consider a punctured Riemann sphere 
\begin{equation}\label{SPHERE}
S_n=\mathbf{S}^2\backslash \{x_1,\ldots, x_n\},
\end{equation} 
where $x_1,\ldots, x_n\in \mathbf{S}^2$ are $n$ points. In this case, $p=0$. Assume that $n\geq 4$.  Hironaka--Kin \cite{H-K} used an explicit example to give an upper bound for $L(\mbox{Mod}_{S_n})$. More precisely, they showed that 
\begin{equation}\label{PEN2}
L(\mbox{Mod}_{S_n})<\frac{2 \log\left(2+\sqrt{3}\right)}{n-3}.
\end{equation} 
By combining (\ref{PEN1}) and (\ref{PEN2}), we obtain 
\begin{equation}\label{PEN3}
\frac{\log 2}{4n-12}< L(\mbox{Mod}_{S_n})< \frac{2 \log\left(2+\sqrt{3}\right)}{n-3}.
\end{equation}
Let $\mathscr{F}$ be the set of all pseudo-Anosov maps on $S_n$ isotopic to the identity on $S_{n-1}$ as $x_n$ is filled in. We know \cite{Kr} that $\mathscr{F}$ is nonempty and contains infinitely many elements. In contrast to (\ref{PEN3}), we prove the following result in this paper:
\begin{thm}\label{T2}
Let $S_n$ be a punctured Riemann spheres with $n\geq 4$ punctures. Then we have the following inequality: 
$\log \left( 2n-5 \right)\leq L(\mathscr{F})\leq \log \left( 2n^2-6n+3\right)$.
\end{thm}
\noindent {\em Remark. } By Theorem 1.10 of Dowdall \cite{D}, for a compact Riemann surface of genus $p\geq 2$, $\tfrac{1}{5}\ \log (2p)\leq L(\mathscr{F})< p\ \log (11)$. Our argument improves the lower bound and shows that for any Riemann surface of type $(p,n)$ with $3p+n>3$, $L(\mathscr{F})\geq \log \left(4p+2n-5 \right)$ if $n>1$, and $L(\mathscr{F})\geq \log \left( 4p-1 \right)$ if $n=1$. See Corollary 3.2.

\medskip
 
Theorem \ref{T2} can be used to study pseudo-Anosov maps with small dilatations in some special cases.
It is well known \cite{Th} that for any two simple closed geodesics $a$ and $b$ on a Riemann surface $S$ of type $(p,n)$ with $3p+n>3$, if $(a,b)$ fills $S$ (in the sense that the complement $S\backslash \{a,b\}$ consists of disks and possibly once punctured disks), then the products 
\begin{equation}\label{PP}
t_{a}^{r}\circ t_b^{-s}
\end{equation} 
for all positive integers $r$ and $s$ are pseudo-Anosov, where $t_c$ denotes the positive Dehn twist along $c$. In \cite{H-L}, Hubert--Lanneau proved that some pseudo-Anosov maps are not of the forms (\ref{PP}). In \cite{Lei}, Leininger showed that if $f$ is such that 
log $\lambda(f)=L(\mbox{Mod}_S)$, then $f$ is not of the form (\ref{PP}). As an application of Theorem \ref{T2}, we prove the following result.
\begin{thm}\label{T4}
Let $S_n$ be as given in $(\ref{SPHERE})$. Assume that $n\geq 7$. Let $f\in \mathscr{F}$ be such that $\log \lambda(f)=L(\mathscr{F})$.  Then $f$ cannot be represented by a product of two Dehn twists along any two simple closed geodesics on $S_n$.
\end{thm}


Although we have estimation (\ref{PEN3}), the exact values of $L(\mbox{Mod}_{S_n})$ are known only for a few simple cases, which are outlined as follows (see \cite{H-S} for more details). We denote by $\sigma_i$ the braid shown in Figure 1, which represents a generator in the Artin braid group. See Birman \cite{Bir} for more details. 

\bigskip
\medskip

\unitlength 1mm 
\linethickness{0.4pt}
\ifx\plotpoint\undefined\newsavebox{\plotpoint}\fi 
\begin{picture}(91.5,37)(0,0)
\put(42,35){\line(0,-1){30.5}}
\put(54.25,35){\line(0,-1){30}}
\put(80,35.25){\line(0,-1){30.25}}
\put(91.5,35.25){\line(0,-1){30.25}}
\put(63.25,35){\line(0,-1){11.75}}
\multiput(63.25,23.25)(.033695652,-.039130435){230}{\line(0,-1){.039130435}}
\put(71,14.25){\line(0,-1){9}}
\put(71.25,34.5){\line(0,-1){11}}
\put(71.25,23.5){\line(-2,-3){3}}
\multiput(66.5,17.25)(-.033505155,-.038659794){97}{\line(0,-1){.038659794}}
\put(63.25,13.5){\line(0,-1){8.5}}
\put(42,36.75){\makebox(0,0)[cc]{1}}
\put(54.25,37){\makebox(0,0)[cc]{$i-1$}}
\put(63,36.75){\makebox(0,0)[cc]{$i$}}
\put(71,36.25){\makebox(0,0)[cc]{$i+1$}}
\put(80,36.5){\makebox(0,0)[cc]{$i+2$}}
\put(91.5,37){\makebox(0,0)[cc]{$n$}}
\multiput(44.68,20.93)(.90625,0){9}{{\rule{.4pt}{.4pt}}}
\multiput(82.18,20.68)(.88889,0){10}{{\rule{.4pt}{.4pt}}}
\put(67.5,.25){\makebox(0,0)[cc]{Figure 1}}
\end{picture}

\bigskip

The 3-braid $\sigma_2\sigma_1^{-1}$ is proved (Matsuoka \cite{M} and Handel \cite{H}) to be pseudo-Anosov and has the minimum dilatation among pseudo-Anosov braids. If we identify the boundary $\partial \Delta$, where $\Delta=\{z:|z|<1\}$, with a point, and identify those $\sigma_i$ with the Magnus generators (denoted by $\sigma_i$ also) interchanging $x_i$ and $x_{i+1}$, we obtain a pseudo-Anosov map $\sigma=\sigma_2\sigma_1^{-1}$ on a $4$-punctured sphere $S_4$. Similar constructions are also valid when $n=5,6$. It was shown in Ham--Song \cite{H-S} that $\sigma$ is pseudo-Anosov and has the minimum dilatation among pseudo-Anosov maps, where $\sigma=\sigma_3\sigma_2\sigma_1^{-1}$ if $n=5$ and $\sigma=\sigma_1\sigma_2\sigma_3\sigma_4\sigma_1\sigma_2$ if $n=6$. 

We remark that the above instances are the only known pseudo-Anosov maps with the minimum dilatations. By examining these examples we find that $\sigma$ has the following common properties: (i) $\sigma(x_i)\neq x_i$ for $1\leq i\leq n-1$, (ii) $\sigma$ fixes $x_n$ and (iii) when the puncture $x_n$ is filled in, the mapping class $\sigma^q$ for any $q\leq n$ defines a nontrivial mapping class on $S_{n-1}$. The following result says that on any punctured Riemann sphere $S_n$, $n\geq 4$, these properties remain valid for pseudo-Anosov maps with the minimum dilatations even though these pseudo-Anosov maps are still unknown.
\begin{thm}\label{T3}
Let $S_n$ be as given in $(\ref{SPHERE})$.
Let $n\geq 4$. Let $f_n:S_n\rightarrow S_n$ be a pseudo-Anosov map such that $L(\mbox{Mod}_{S_n})=\log \lambda(f_n)$. Then $f_n$ is not isotopic to the identity on $S\cup \{x_i\}$ whenever $f_n(x_i)=x_i$. Furthermore, if $n\geq 16$, then for any integer $q\leq n$, $f_n^q$ is not isotopic to the identity on $S\cup \{x_i\}$ whenever $f_n^q(x_i)=x_i$. 
\end{thm}  

This paper is organized as follows. In Section 2 we present some background materials needed in our estimations of $\lambda(f)$ for elements of $\mathscr{F}$. A lower bound and an upper bound for $\lambda(f)$, $f\in \mathscr{F}$, in punctured Riemann sphere cases will be given in Section 3 and Section 4 respectively. The proof of Theorem 1.1 also appears in these two sections. In Section 5, we prove Theorem 1.2 and Theorem 1.3. In Appendix, we sketch a new proof of Theorem \ref{K} by using methods different from Kra \cite{Kr}.  

\bigskip

\noindent {\bf Acknowledgment. } I am grateful to the referees for their valuable comments and thoughtful suggestions. 

\section{Background and Notation}
\setcounter{equation}{0}

Let $S$ be a Riemann surface of type $(p,n)$. Assume that $3p+n>3$ and $n\geq 1$. Let $x$ be a puncture. Write $\tilde{S}=S\cup \{x\}$. Then $\tilde{S}$ is of type $(p,n-1)$. Let $\mathscr{J}(\tilde{S})$ be the set of conformal structures on $\tilde{S}$. The Teichm\"{u}ller space $T(\tilde{S})$ is the quotient space $\mathscr{J}(\tilde{S})/\! \sim$, where two conformal structures $\nu_1$ and $\nu_2$ are considered equivalent if there is a conformal map $h:\nu_1(\tilde{S})\rightarrow \nu_2(\tilde{S})$ such that $h$ is isotopic to the identity on the underlying surface $\tilde{S}$. $T(\tilde{S})$ is a complex manifold with dimension $3p-3+n$. See \cite{A-B,Bers1, Bers3,Kr} for an detailed account of Teichm\"{u}ller theory.   

Let $V(\tilde{S})\rightarrow T(\tilde{S})$ be the fiber bundle so that the fiber over a point $\tau\in T(\tilde{S})$ is a Riemann surface representing the point $\tau$. Let $F(\tilde{S})\rightarrow V(\tilde{S})$ be the universal covering space. We thus have the natural projection $\pi:F(\tilde{S})\rightarrow T(\tilde{S})$. An important theorem of Bers \cite{Bers1} states that there exists a biholomorphic map $\varphi:F(\tilde{S})\rightarrow T(S)$ so that $\iota\circ \varphi=\pi$, where $\iota:T(S)\rightarrow T(\tilde{S})$ is the natural projection defined by forgetting the puncture $x$. 

Let $\mathbf{H}$ denote the unit disk equipped with the hyperbolic metric with constant negative curvature $-1$. Let $\varrho:\mathbf{H}\rightarrow \tilde{S}$ be the universal covering map with a covering group $G$. Fix a point $x\in \tilde{S}$, then the fundamental group $\pi_1(\tilde{S},x)$ is isomorphic to $G$. Thus every filling closed geodesic $\gamma$ on $\tilde{S}$ (that is, $\tilde{S}\backslash \gamma$ is the union of disks and once punctured disks) determines a family of hyperbolic elements $g$ of $G$. These elements can act on $F(\tilde{S})$ as holomorphic automorphisms. Now with the aid of the Bers isomorphism $\varphi$, we obtain a family $\mathscr{F}_{\gamma}$ of mapping classes $g^*=\varphi\circ g\circ \varphi^{-1}$ on $S$. All elements of $\mathscr{F}_{\gamma}$ fix the puncture $x$ and project to the trivial mapping class on $\tilde{S}$.  Let $\mathscr{F}$ be the set of pseudo-Anosov mapping classes of $S$ isotopic to the identity on $\tilde{S}$. By Theorem 2 of Kra \cite{Kr}, we have  $\mathscr{F}_{\gamma}\subset \mathscr{F}$. 

Elements in $\mathscr{F}_{\gamma}$ can also be fully described by the geometric and topological terms (\cite{D}). Each filling closed geodesic $\gamma$ can be parametrized as $\gamma:[0,1]\rightarrow \tilde{S}$ such that $x=\gamma(0)=\gamma(1)$. This parametrization  determines an isotopy $I:\{x\}\times [0,1]\rightarrow \tilde{S}$ given by $I(x,\cdot)=\gamma(t)$ even if $\gamma$ is self-intersecting. Furthermore, the isotopy $I(x,\cdot)$ can be extended (via the identity) to an isotopy 
\begin{equation}\label{EXT}
\tilde{I}:\tilde{S}\times [0,1]\rightarrow \tilde{S}.
\end{equation}
The resulting map $f=\tilde{I}(\cdot,1)$ fixes the base point $x$ and hence defines a map on $S=\tilde{S}\backslash \{x\}$. In the literature, $f$ is called a point-pushing homeomorphism corresponding to the curve $\gamma$. 

To see how an element $f\in \mathscr{F}$ determines a filling closed geodesic $\gamma$, we let $F:\tilde{S}\times [0,1]\rightarrow \tilde{S}$ denote the corresponding isotopy between $f$ and the identity. That is,
$$
F(\cdot,0)=f \ \  \mbox{and}\ \  F(\cdot,1)=id.
$$
Then $F(x,t)$, $t\in [0,1]$, defines a Jordan curve. In its homotopy class, a geodesic representative is a closed geodesic $\gamma$. Theorem 2 of \cite{Kr} also implies that $\gamma$ is filling, and thus (by combining results of Bers \cite{Bers1} and Birman \cite{Bir}),  $f\in \mathscr{F}_{\gamma}$. In summary, we have
\begin{thm}[Kra \cite{Kr}]\label{K}
$\mathscr{F}=\bigcup \left\{ \mathscr{F}_{\gamma}: \gamma \ \mbox{is filling on}\ \tilde{S}\right\}$.
\end{thm}
See also \cite{F-M} for different treatments of this result.  In Appendix, we will give another proof of Theorem \ref{K}.

Let $a$ and $b$ be arbitrary simple closed geodesics on $S$. We define the geometric intersection number $i(a,b)$ to be 
$$
\mbox{inf }\{a'\cap b'\},
$$
where $a'$ and $b'$ are homotopic to $a$ and $b$, respectively. The following theorem is due to Thurston. See \cite{FLP} for a proof.
\begin{thm}[Thurston \cite{FLP}]\label{TT}
Let $f$ be a pseudo-Anosov map with dilatation $\lambda(f)$. Then for any two simple closed geodesics $a$ and $b$ on $S$, as $k\rightarrow +\infty$, the ratio
$$
\frac{i(f^k(a),b)}{\lambda(f)^k}\rightarrow \kappa,
$$
where $\kappa$ is the number with $0<\kappa<+\infty$. 
\end{thm}
In \cite{D} Dowdall used Theorem \ref{TT} to estimate the dilatation $\lambda(f)$ with respect to the intersection numbers of two curves. The same idea will be used in the rest of the paper. 

\section{A lower bound for the least dilatations of certain pseudo-Anosov maps}
\setcounter{equation}{0}

In this section, we assume that $S$ is a Riemann surface of type $(p,n)$ which contains at least one puncture $x$. As discussed in Section 2, every element $f\in \mathscr{F}$ is determined by a filling closed geodesic $\gamma$ on $S\cup \{x\}$ and vice versa. Since $\gamma$ is a filling curve, it must have self-intersection points. Let $i_{\gamma}$ denote the number of self-intersection points of $\gamma$. It was shown in \cite{D} that for any element $f\in \mathscr{F}$ with the corresponding filling closed geodesic $\gamma$, we have $\left(1+i_{\gamma} \right)^{1/5} \leq \lambda(f) \leq 9^{i_{\gamma}}$. 

Let $\mathscr{F}_0$ be the set of all pseudo-Anosov maps on $S$ obtained from a primitive filling closed geodesic $\gamma$ on $S\cup \{x\}$. Here by a primitive curve we mean that $\gamma$  cannot be expressed as a power of another curve. Then every element $f$ in $\mathscr{F}$ is conjugate in the fundamental group $\pi_1(S\cup \{x\})$ to a power of an element $f_0$ in $\mathscr{F}_0$.  The aim of this section is to give a better lower bound estimation of $\lambda(f)$ for $f\in \mathscr{F}_0$ with respect to the self-intersection number $i_{\gamma}$. We will later on prove the following result.
\begin{thm}\label{T1}
Let $f\in \mathscr{F}_0$ be determined by a primitive filling closed geodesic $\gamma$. Then the dilatation $\lambda(f)$ satisfies $\lambda(f)\geq 1+2 i_{\gamma}$.
\end{thm}
Write $k=i_{\gamma}$ and let $\left\{P_1,P_2,\ldots, P_k\right\}$ be the set of self-intersection points of $\gamma$. Observe that via the point-pushing isotopy $\tilde{I}(t,\cdot)$ given by (\ref{EXT}),  $x_t=\tilde{I}(t,x)$ can travel along $\gamma$ as many times as possible, and on each loop, $x_t$ goes through each intersection point exactly twice.  

Let $c$ be a simple closed geodesic on $S$ that also defines a nontrivial geodesic (also identifies as $c$) on $\tilde{S}$. Since $\gamma$ is a filling closed geodesic on $\tilde{S}$, $\gamma$ must intersect $c$. This means that in the isotopy $\tilde{I}(t,\cdot)$, $0\leq t\leq 1$, the point 
$x_t$ does not carry any strands until it moves across the curve $c$, then $x_t$ starts capturing more and more strands as it passes through each intersection point $P_j$. Here by a strand we mean a portion of the image $\tilde{I}(t,c)$ of $c$ surrounding the point $x_t$ during the point-pushing process. As an illustration, Figure 2 below shows that $x_t$ at the moment $t$ carries 4 strands, although there may exist some other leaves (such as $c_1$ and $c_2$ in Figure 2) near $x_t$.  


\bigskip
\smallskip

\unitlength 1mm 
\linethickness{0.4pt}
\ifx\plotpoint\undefined\newsavebox{\plotpoint}\fi 
\begin{picture}(92,33.25)(0,0)
\put(42.75,33.25){\line(1,0){47.25}}
\put(42.75,32.25){\line(1,0){47.25}}
\put(42.75,10){\line(1,0){47}}
\put(42.75,9){\line(1,0){47}}
\put(45,27){\line(1,0){38.75}}
\put(83.75,27){\line(0,-1){13}}
\put(45.25,25.75){\line(1,0){37.5}}
\put(82.75,25.75){\line(0,-1){10.75}}
\put(83.75,14.25){\line(-1,0){39}}
\put(82.75,15.5){\line(-1,0){38}}
\put(45.25,24.75){\line(1,0){36.5}}
\put(81.75,24.75){\line(0,-1){8}}
\put(81.75,16.75){\line(-1,0){36.75}}
\put(80.75,23.75){\line(0,-1){6}}
\put(80.75,17.75){\line(-1,0){35.75}}
\put(88.25,30){\makebox(0,0)[cc]{$c_1$}}
\put(87.5,12){\makebox(0,0)[cc]{$c_2$}}
\put(80.75,24){\line(-1,0){35.75}}
\put(42,21){\vector(1,0){50}}
\put(77.25,21){\circle*{1.803}}
\put(74.75,19.5){\makebox(0,0)[cc]{$x_t$}}
\put(89.75,18.25){\makebox(0,0)[cc]{$\gamma$}}
\put(66.25,3.25){\makebox(0,0)[cc]{Figure 2}}
\end{picture}
 
Throughout the article, $f^q(c)$ denotes the geodesic representative in the homotopy class of the image curve of $c$ under the map $f^q$ for an integer $q$. Note that the point-pushing process gives rise to a deformation of the curve $c$. It is clear that $f^q(c)$ is homotopic to the deformation image of $c$ when $x_t$ completes its $q$-th cycle and returns to its original position. 

Let $\#\{f^q(c)\cap c\}$ denote the set of points of intersection between $f^q(c)$ and $c$.
We need to examine more carefully the points of intersection in $\#\{f^q(c)\cap c\}$ when $q$ is large.  Recall that $i(f^q(c),c)$ is the geometric intersection number of $f^q(c)$ and $c$. 

Since our aim is to seek the least value of the intersection number $i(f^{m+1}(c),c)$, the worst scenario is that $c$ intersects $\gamma$ exactly once and in addition, during the first pass through $\gamma$, the point $x_t$ moves across $c$ right before $x_t$ returns to its original position. In this case, $x_t$ does not carry any strands most the time in the first pass, and captures only one strand right before $x_t$ returns to its original position. It follows that $x_t$ carries only one strand at the end of the first trip through $\gamma$. 

As a matter of fact, $f(c)$ and $c$ are disjoint and isotopic to each other as $x$ is filled in. In other words, $\{f(c),c\}$ are boundary components of an once punctured cylinder. It turns out that
\begin{equation}\label{BN}
i(f(c),c)=0.
\end{equation}
We need the following result.
\begin{lem}\label{L2}
Let $f\in \mathscr{F}_0$ be determined by a primitive filling closed geodesic $\gamma$ on $\tilde{S}$. Then for any integer $m\geq 1$, 
\begin{equation}\label{OO4}
i(f^{1+m}(c),c)\geq 2\ \sum_{j=1}^m\left( 1+2 i_{\gamma}  \right)^j.
\end{equation}
\end{lem}
\begin{proof}
Suppose at the beginning of the $m$-th pass through $\gamma$, $m\geq 1$, the point $x_t$ carries $s$ strands. Write 
$$
P_j=\gamma(t_j^1)=\gamma(t_j^2)
$$
for $1\leq j\leq k$. Obviously, for each $j$ we have $0<t_j^1<t_j^2<1$.

Denote by $\mu_m(t)$ the number of strands carried by $x_t$ during the $m$-th trip through $\gamma$ at time $t$. It is obvious that $\mu_m(t)$ is a nondecreasing function on $[0,1]$ with non-negative integer values. Let 
$$
\mu_m^+(t)=\mu_m(t+\varepsilon)\ \  \mbox{and}\ \  \mu_m^-(t)=\mu_m(t-\varepsilon)
$$ 
for a small positive number $\varepsilon$. We claim that for every $j$, $1\leq j\leq k$,
\begin{equation}\label{P0}
\mu_m^+(t_j^2)\geq \mu_m^-(t_j^2)+2\ \mu_m(t_j^1).
\end{equation}
Indeed, as $x_t$ travels through the intersection point $P_j=\gamma(t_{j}^1)$, the point $x_t$ carries at least $s$ strands (Figure 3 shows that $x_t$ carries $s=2$ strands and tends to travel through $P_j$). Then after going through $P_j=\gamma(t_{j}^1)$, at least $2s$ leaves near $P_j$ have been recorded as shown in Figure 4.

\bigskip

\unitlength 1mm 
\linethickness{0.4pt}
\ifx\plotpoint\undefined\newsavebox{\plotpoint}\fi 
\begin{picture}(116,29.25)(0,0)
\put(20.75,22.25){\line(1,0){24}}
\put(44.75,22.25){\line(0,-1){7.5}}
\put(44.75,14.75){\line(-1,0){24.25}}
\put(20.75,21.5){\line(1,0){23.5}}
\put(44.25,21.5){\line(0,-1){6}}
\put(44.25,15.5){\line(-1,0){23.75}}
\put(19.25,18.25){\line(1,0){37}}
\put(48.25,29.25){\line(0,-1){19.25}}
\put(41.25,18){\circle*{1.118}}
\put(39.5,20){\makebox(0,0)[cc]{$x_t$}}
\put(50.25,16){\makebox(0,0)[cc]{$P_j$}}
\put(50.75,27){\makebox(0,0)[cc]{$\gamma$}}
\put(38.75,5.75){\makebox(0,0)[cc]{Figure 3}}
\put(79,18.25){\line(1,0){37}}
\put(96.75,28){\line(0,-1){17.5}}
\put(81,22.5){\line(1,0){33.75}}
\put(81,23.25){\line(1,0){33.75}}
\put(81.5,14){\line(1,0){33.25}}
\put(81.5,13.25){\line(1,0){33.25}}
\put(99,16.25){\makebox(0,0)[cc]{$P_j$}}
\put(99.25,26.75){\makebox(0,0)[cc]{$\gamma$}}
\put(113.25,16){\makebox(0,0)[cc]{$\gamma$}}
\put(98.5,6.25){\makebox(0,0)[cc]{Figure 4}}
\end{picture}

Now the point $x_t$ continues to travel along $\gamma$. Right before $x_t$ returns to $P_j$, that is, when $t=t_j^2-\varepsilon$, $x_t$ carries $\mu_m^-(t_{j}^2)$ strands (Figure 5 shows the situation that $x_t$ currently carries 4 strands and tends to travel through $P_j=\gamma(t_j^2)$). 

Observe that $x_t$, $t_j^2\leq t<t_j^2+\varepsilon$, captures at least $2s$ more strands than it does when $t_j^2-\varepsilon < t < t_j^2$. This means that $x_t$ carries at least $2s+\mu_m^-(t_j^2)$ strands when $t_j^2\leq t<t_{j}^2+\varepsilon$. (Figure 6 depicts the situation that $x_t$ has just passed through $P_j=\gamma(t_{j}^2)$ and captures those strands obtained from Figure 4). We conclude that (\ref{P0}) is satisfied. 

\bigskip

\unitlength 1mm 
\linethickness{0.4pt}
\ifx\plotpoint\undefined\newsavebox{\plotpoint}\fi 
\begin{picture}(120.25,38.75)(0,0)
\put(17.75,32){\line(1,0){39}}
\put(35.75,38.25){\line(0,-1){26}}
\put(38,29.25){\makebox(0,0)[cc]{$P_j$}}
\put(18.5,27.75){\line(1,0){37.75}}
\put(18.5,27){\line(1,0){37.75}}
\put(18.5,35.25){\line(1,0){38}}
\put(18.5,35.75){\line(1,0){38}}
\put(32.25,13){\line(0,1){11.5}}
\put(32.25,24.5){\line(1,0){7.25}}
\put(39.5,24.5){\line(0,-1){11.75}}
\put(31.25,13){\line(0,1){12.25}}
\put(31.25,25.25){\line(1,0){9.25}}
\put(40.5,25.25){\line(0,-1){12.5}}
\put(30.25,13){\line(0,1){13}}
\put(30.25,26){\line(1,0){11.25}}
\put(41.5,26){\line(0,-1){13.25}}
\put(33,13){\line(0,1){11}}
\put(33,24){\line(1,0){6}}
\put(39,24){\line(0,-1){11.25}}
\put(35.75,21.5){\circle*{1}}
\put(37.5,18.5){\makebox(0,0)[cc]{$x_t$}}
\put(36.25,5.25){\makebox(0,0)[cc]{Figure 5}}
\put(54,29.75){\makebox(0,0)[cc]{$\gamma$}}
\put(92.25,12.5){\line(0,1){14.5}}
\put(92.25,27){\line(1,0){6}}
\put(98.25,27){\line(0,-1){15}}
\put(91.5,12.25){\line(0,1){15.5}}
\put(91.5,27.75){\line(1,0){7.5}}
\put(99,27.75){\line(0,-1){15.75}}
\put(90.75,12.25){\line(0,1){16.25}}
\put(90.75,28.5){\line(1,0){9}}
\put(99.75,28.5){\line(0,-1){16.5}}
\put(90,12.25){\line(0,1){17}}
\put(90,29.25){\line(1,0){10.5}}
\put(100.5,29.25){\line(0,-1){17.25}}
\put(74.25,14.25){\line(1,0){13.75}}
\put(88,14.25){\line(0,1){17}}
\put(88,31.25){\line(1,0){14.5}}
\put(102.5,31.25){\line(0,-1){17.25}}
\put(102.5,14){\line(1,0){13.25}}
\put(74.25,15){\line(1,0){13}}
\put(87.25,15){\line(0,1){17}}
\put(87.25,32){\line(1,0){16.25}}
\put(103.5,32){\line(0,-1){17}}
\put(103.5,15){\line(1,0){12.5}}
\put(74.25,17){\line(1,0){11}}
\put(85.25,17){\line(0,1){18}}
\put(85.25,35){\line(1,0){20.5}}
\put(105.75,35){\line(0,-1){18}}
\put(105.75,17){\line(1,0){10.25}}
\put(73.75,17.75){\line(1,0){10.75}}
\put(84.5,17.75){\line(0,1){18}}
\put(84.5,35.75){\line(1,0){22}}
\put(106.5,35.75){\line(0,-1){18}}
\put(106.5,18){\line(0,-1){.25}}
\put(106.5,17.75){\line(1,0){9.5}}
\put(95.5,38){\line(0,-1){26.75}}
\put(95.5,24){\circle*{1.118}}
\put(98.25,38.75){\makebox(0,0)[cc]{$\gamma$}}
\put(95.5,5.5){\makebox(0,0)[cc]{Figure 6}}
\put(94.5,21.75){\makebox(0,0)[cc]{$x_t$}}
\put(118,17.25){\makebox(0,0)[cc]{$\gamma$}}
\put(93.75,17){\makebox(0,0)[cc]{$P_j$}}
\put(72.25,16){\line(1,0){48}}
\end{picture}

Let $k_0$ be the smallest index such that 
$$
t_{k_0}^2=\mbox{min }\{t_j^2; \  \mbox{for} \ 1\leq j\leq k\}.
$$
It is clear that  $\mu_m^+(t_{k_0}^2)\geq s+2s$. From (\ref{P0}), we see that for $j\neq k_0$, 
$$
\mu_m^+(t_j^2)-\mu_m^-(t_j^2)\geq 2s.
$$
It follows that when $x_t$ completes its $m$-th cycle,  the total number of strands carried by $x_t$ is larger than 
\begin{equation}\label{P1}
2s(k-1)+(2s+s)=s(2k+1)=s(2i_{\gamma}+1). 
\end{equation}
By induction hypothesis, suppose (\ref{OO4}) holds for an integer $m\geq 1$. According to (\ref{P1}), at the beginning of the $(m+1)$-th cycle, $x_t$ carries at least $(1+2i_{\gamma})^m$ strands. At the end of $(m+1)$-th cycle, $x_t$ carries at least  
$$
(1+2i_{\gamma})^m(1+2i_{\gamma})=(1+2i_{\gamma})^{m+1}
$$ 
strands. Since each strand contributes two points of intersection in $\#\{f^{1+(m+1)}(c)\cap c\}$, we conclude that 
\begin{equation}\label{OO5}
i(f^{1+(m+1)}(c),c)\geq i(f^{1+m}(c),c)+2(1+2i_{\gamma})^{m+1}\geq 2\ \sum_{j=1}^{m+1}\left( 1+2i_{\gamma}  \right)^j. 
\end{equation}
This proves the lemma. 
\end{proof}

\noindent {\em Proof of Theorem $\ref{T1}$: } Let $\alpha=1+2 i_{\gamma}$. From Lemma \ref{L2}, we obtain
\begin{equation}\label{F0}
i(f^{1+m}(c),c)\geq 2\alpha\ \frac{\alpha^{m}-1}{\alpha -1},
\end{equation} 
which implies that
$$
\frac{i(f^{1+m}(c),c)}{\lambda(f)^m}\geq \frac{2\alpha}{\lambda(f)^m}\frac{\alpha^{m}-1}{\alpha - 1}.
$$
Therefore,
\begin{equation}\label{C0}
\frac{i(f^{1+m}(c),c)}{\lambda(f)^m}\geq \frac{2\alpha}{\alpha-1}\left\{\left(\frac{\alpha}{\lambda(f)}   \right)^m - \frac{1}{\lambda(f)^m}\right\}.
\end{equation}
Setting $b=f(c)$, the left hand side of (\ref{C0}) becomes $i(f^m(b),c)/\lambda(f)^m$. From Theorem \ref{TT}, we get 
\begin{equation}\label{CC1}
\frac{i(f^m(b),c)}{\lambda(f)^m}\rightarrow \kappa \  \ \mbox{as}\ m\rightarrow +\infty,
\end{equation}
where $\kappa\in (0,+\infty)$.  Since $\lambda(f)>1$, $1/\lambda(f)^m\rightarrow 0$ as $m\rightarrow +\infty$. Now from (\ref{C0}) and (\ref{CC1}) it follows that 
$\alpha/\lambda(f)\leq 1$, which says that 
$$
\lambda(f)\geq \alpha = 1+2  i_{\gamma}.
$$
This proves Theorem 3.1. \ \ \ \ \ \ \ \ \ \ \ \ \ \ \ \ \ \ \ \ \ \ \ \ \ \ \ \ \ \ \ \ \ \ \ \ \ \ \ \ \ \ \ \ \ \ \ \ \ \ \ \ \ \ \ \ \ \ \ \ \ \ \ \ \ \ \ \ \ \ \ \ \ \ \ \ \ \ \ \ \ \ \ \ \ \ \ \ \ $\Box$

\medskip

As a direct consequence of Theorem \ref{T1}, we have the following corollary:
\begin{cor}\label{C1}
Let $S$ be a Riemann surface of type $(p,n)$. Suppose that $3p+n>3$ and $n\geq 1$. Then 
$L(\mathscr{F})\geq \mbox{log}\left( 4p+2n-5 \right)$ if $n>1$, and 
$L(\mathscr{F})\geq \mbox{log}\left( 4p-1 \right)$ if $n=1$.
\end{cor}
\begin{proof}
We first prove that $L(\mathscr{F}_0)\geq \mbox{log}\left( 4p+2n-5 \right)$ if $n>1$, and 
$L(\mathscr{F}_0)\geq \mbox{log}\left( 4p-1 \right)$ if $n=1$.

Choose $f\in \mathscr{F}_0$, and let $f$ be determined by a primitive filling closed geodesic $\gamma$.  The curve $\gamma$ can be thought of as a 4-valence graph on $\tilde{S}$. By Euler characteristic, we obtain 
$$
2-2p=V+F-E,
$$
where $V=i_{\gamma}$ is the number of vertices, $E$ is the number of edges, and $F$ is the number of faces. In our situation, $2V=E$. Note that $\gamma$ is on the surface $S\cup \{x\}$ which is of type $(p,n-1)$.  We see that $F\geq n-1$ if $n\geq 2$, and $F\geq 1$ if  $n=1$.  

 It turns out that $2-2p\geq -i_{\gamma}+n$ if $n\geq 2$; and $2-2p\geq -i_{\gamma}+1$ if $n=1$. Hence 
$i_{\gamma}\geq 2p-3+n$ if $n\geq 2$, and $i_{\gamma}\geq 2p-1$ 
if $n=1$. It follows from Theorem 3.1 that 
$$
\lambda(f)\geq 1+2i_{\gamma}\geq 1+2(2p-3+n)=4p+2n-5
$$
if $n\geq 2$, and 
$$
\lambda(f)\geq 1+2i_{\gamma}\geq 1+2(2p-1)=4p-1
$$
if $n=1$. So the corollary is proved when $f\in \mathscr{F}_0$. 

For the general case, we notice that $\mathscr{F}_0\subset \mathscr{F}$. As such $L(\mathscr{F})\leq L(\mathscr{F}_0)$. On the other hand, for any element $f\in \mathscr{F}$, there is an element $f_0\in \mathscr{F}_0$, an integer $m\geq 1$, and an element $g\in G$ such that
\begin{equation}\label{BV}
f=g^*\circ f_0^m \circ (g^*)^{-1}, 
\end{equation}
where $g^*\in \mbox{Mod}_S^x$ is the corresponding element of $g$ under the Bers isomorphism \cite{Bers1}. Clearly, from (\ref{BV}) we obtain
$$ 
\lambda(f)=\lambda(f_0^m)=\lambda(f_0)^m.
$$
Since $m\geq 1$ and $\lambda(f_0)>1$, we conclude that $\lambda(f)\geq \lambda(f_0)$. Therefore, $L(\mathscr{F})=L(\mathscr{F}_0)$, and so the corollary still holds for the general case. 
\end{proof}

\medskip

As a special case of Corollary \ref{C1},  Theorem \ref{T2} in lower bound case is proved. 

\section{An upper bound for the least dilatations of certain pseudo-Anosov maps on punctured spheres }
\setcounter{equation}{0}

On the sphere $S_{n-1}=\mathbf{S}^2\backslash \{x_1,\ldots, x_{n-1}\}$, $n\geq 4$, there is a primitive filling closed geodesic $\gamma$ drawn in Figure 7.  


\unitlength 1mm 
\linethickness{0.4pt}
\ifx\plotpoint\undefined\newsavebox{\plotpoint}\fi 
\begin{picture}(122.75,42.875)(0,0)
\put(24.75,24.125){\oval(24,2.25)[]}
\put(46,24){\circle*{.5}}
\put(64,24){\circle*{.5}}
\put(87.75,24){\circle*{.5}}
\put(103.75,24){\circle*{.5}}
\qbezier(24.75,24)(35.125,5.875)(46,24.25)
\qbezier(24.75,24)(35.25,42.875)(45.75,24.25)
\qbezier(45.75,24.25)(55.125,42.75)(64,24.25)
\qbezier(45.75,24.25)(55.25,6.375)(63.75,24)
\qbezier(87.5,24)(95.5,41.875)(103.5,24.25)
\qbezier(87.5,24.25)(95.5,6.875)(103.5,24)
\qbezier(103.5,24.25)(111.625,41.25)(119.25,24.25)
\qbezier(103.5,24.25)(111.625,7.375)(119.25,24)
\qbezier(64,24.25)(65.5,29.875)(71,33)
\qbezier(64,24.25)(65.5,17.125)(71,15.5)
\qbezier(87.5,24)(86.5,30)(81.5,33)
\qbezier(87.5,24)(86,18)(81.5,16)
\multiput(72.43,23.68)(.96875,0){9}{{\rule{.4pt}{.4pt}}}
\put(35.25,24){\circle*{.5}}
\put(55.25,24){\circle*{1}}
\put(95.5,24){\circle*{1.118}}
\put(111.5,23.75){\circle*{1}}
\put(14.5,24){\circle*{.707}}
\put(35.25,33.5){\circle*{1}}
\put(35.25,33.25){\circle*{1}}
\put(35.25,21){\makebox(0,0)[cc]{$x_1$}}
\put(55.25,21.25){\makebox(0,0)[cc]{$x_2$}}
\put(35,30.5){\makebox(0,0)[cc]{$x_0$}}
\put(48.5,24){\makebox(0,0)[cc]{$P_1$}}
\put(66.75,24){\makebox(0,0)[cc]{$P_2$}}
\put(83.5,24){\makebox(0,0)[cc]{$P_{k-1}$}}
\put(100.25,24){\makebox(0,0)[cc]{$P_k$}}
\multiput(42.75,29.75)(.0333333,-.1166667){15}{\line(0,-1){.1166667}}
\multiput(43.25,28)(-.0543478,.0326087){23}{\line(-1,0){.0543478}}
\multiput(54.5,15.75)(.1,-.0333333){15}{\line(1,0){.1}}
\multiput(56,15.25)(-.1,-.0333333){15}{\line(-1,0){.1}}
\put(54.5,14.75){\line(-1,0){.25}}
\multiput(69,32.5)(.21875,.03125){8}{\line(1,0){.21875}}
\multiput(70.75,32.75)(-.0333333,-.05){30}{\line(0,-1){.05}}
\multiput(95.25,16.25)(.0652174,-.0326087){23}{\line(1,0){.0652174}}
\multiput(96.75,15.5)(-.0652174,-.0326087){23}{\line(-1,0){.0652174}}
\multiput(111.25,33)(.076087,-.0326087){23}{\line(1,0){.076087}}
\multiput(113,32.25)(-.21875,-.03125){8}{\line(-1,0){.21875}}
\multiput(41.25,18.75)(-.0333333,-.0833333){15}{\line(0,-1){.0833333}}
\put(40.75,17.5){\line(1,0){1}}
\multiput(55.25,34)(-.05,-.0333333){30}{\line(-1,0){.05}}
\multiput(53.75,33)(.1,-.0333333){15}{\line(1,0){.1}}
\multiput(69.75,16.75)(-.25,.03125){8}{\line(-1,0){.25}}
\multiput(67.75,17)(.03289474,-.03947368){38}{\line(0,-1){.03947368}}
\multiput(82.75,17.25)(-.03289474,-.03947368){38}{\line(0,-1){.03947368}}
\multiput(81.5,15.75)(.1333333,.0333333){15}{\line(1,0){.1333333}}
\multiput(95.75,33.5)(-.0583333,-.0333333){30}{\line(-1,0){.0583333}}
\multiput(94,32.5)(.1333333,-.0333333){15}{\line(1,0){.1333333}}
\multiput(112.5,16.5)(-.076087,-.0326087){23}{\line(-1,0){.076087}}
\multiput(110.75,15.75)(.1,-.0333333){15}{\line(1,0){.1}}
\put(21.25,27.25){\makebox(0,0)[cc]{$c$}}
\put(28.25,15.5){\makebox(0,0)[cc]{$\gamma$}}
\put(69.5,6.25){\makebox(0,0)[cc]{Figure 7}}
\put(14.25,20.75){\makebox(0,0)[cc]{$x_{n-1}$}}
\put(111.25,20.75){\makebox(0,0)[cc]{$x_{n-2}$}}
\put(95.25,20.75){\makebox(0,0)[cc]{$x_{n-3}$}}
\put(35.25,24){\circle*{.5}}
\put(35.25,24){\circle*{.707}}
\put(119.25,24){\circle*{1.414}}
\put(119.25,24){\circle*{1.803}}
\put(122.75,24.5){\makebox(0,0)[cc]{$Z$}}
\end{picture}

Note that the pseudo-Anosov map $f$ arising from $\gamma$ does not depend on the choice of the base point $x=x_0$ on $\gamma$. Thus we may assume without loss of generality that the starting point $x=x_0$ is as shown in Figure 7, and the point-pushing isotopy goes through each intersection point in the following order: 
$P_1\rightarrow P_2\rightarrow \ldots, \rightarrow P_k\rightarrow  P_k\rightarrow P_{k-1}\rightarrow  \ldots \rightarrow P_1$ (where $k=i_{\gamma}=n-3$). In other words, if we write
$P_j=\gamma(t_j^1)=\gamma(t_j^2)$ for $1\leq j\leq k$, then 
\begin{equation}\label{ORDER}
0=t_0^1 < t_1^1 < \cdots  <  t_k^1  <  t_k^2 <  \cdots  <  t_1^2 <  t_0^2 =1. 
\end{equation}
The aim of this section is to give the following estimation for $\lambda(f)$:
\begin{thm}\label{T5}
Let $p=0$ and $n\geq 4$. Let $f\in \mathscr{F}_0$ be determined by a primitive filling closed geodesic $\gamma$ that is drawn in Figure $7$. Then 
\begin{equation}\label{Y7}
\lambda(f)\leq 2i_{\gamma}^2+6i_{\gamma}+3= 2n^2-6n+3.
\end{equation}
\end{thm}
To begin, let $c$ be a simple closed geodesic as shown in Figure 7. That is, $c$ encloses $x_1$ and $x_{n-1}$.  From Figure 7, we see that at the end of the first pass through $\gamma$, the number  of strands carried by $x_t$ is two. That is, 
\begin{equation}\label{GG}
\mu_1^+(t_0^2)=2.
\end{equation}
Suppose that $m\geq 1$ and that at the beginning of the $(m+1)$-th trip, $x_t$ carries $s$ strands. This means that when $x_t$ completes its $m$-th trip, $x_t$ carries $s$ strands. So we obtain 
\begin{equation}\label{Y0}
s\ =\ \mu_m^+(t_1^2)\ \geq \ \mu_m^+(t_2^2)\ \geq \ \cdots \ \geq \ \mu_m^+(t_k^2).
\end{equation} 
By the same argument as in (3.3), we see that 
\begin{equation}\label{YY0}
\mu_{m+1}^+(t_1^1)=s+2\mu_{m}^+(t_1^2)\leq 3s.
\end{equation}
From (\ref{Y0}), for $2\leq j\leq k$, we have
\begin{equation} \label{Y1}
\mu_{m+1}^+(t_j^1)\ \leq \ \mu_{m+1}^+(t_{j-1}^1)+2\ \mu_m^+(t_j^2)\ \leq \ (2j+1)s. 
\end{equation}
At this moment, $x_t$ arrives at the labeling point $Z$ on $\gamma$ (Figure 7) and tends to pass through  
each intersection point $P_j$ a second time in the following order: $P_k\rightarrow P_{k-1}\rightarrow  \cdots \cdots \rightarrow P_1$. Then $x_t$ returns to $x_0$. 
Obviously, 
\begin{equation}\label{PPP0}
\mu_{m+1}^+(t_k^2)\ \leq \ 3s(2k+1)=6sk+3s.
\end{equation}
Note that 
\begin{equation}\label{PPP1}
\mu_{m+1}^-(t_{k-j}^2)=\mu^+_{m+1}(t_{k-j+1}^2).
\end{equation}
Now (\ref{Y1}), (\ref{PPP0}), and (\ref{PPP1}) combine to yield 
\begin{eqnarray*}
\mu_{m+1}^+(t_{k-1}^2)\ \leq  \ \mu_{m+1}^+(t_k^2)+2\mu_{m+1}(t_{k-1}^1)\leq  \ \mu_{m+1}^+(t_k^2)+2\mu_{m+1}^+(t_{k-1}^1) \leq \\
\leq 3s\left(2k+1\right)+2\left(s+2(k-1)s\right) = 10sk+s,  \ \ \ \ \ \ \ \ \ \ \ \ \ \ \ \ \ \
\end{eqnarray*}
and in general, induction hypothesis shows that for each $1\leq j\leq k$,
\begin{eqnarray}\label{Y2}
\mu_{m+1}^+(t_{k-j}^2)\ \leq \ \mu_{m+1}^+(t_{k-j+1}^2)+2\ \mu_{m+1}^+(t_{k-j}^1)\leq \\
 \leq \ sk \left(6+4j)+s(3-2j^2\right). \ \ \ \ \ \ \ \ \ \ \nonumber
\end{eqnarray}
We see that as $x_t$ completes its $(m+1)$-th trip, 
$x_t$ carries $\mu_{m+1}^+(t_0^2)$  strands that satisfies the following inequality (obtained by setting $j=k$ in (4.8)):
\begin{equation}\label{Y3}
\mu_{m+1}^+(t_0^2)\ \leq \ sk (6+4k)+s (3-2k^2)=s (2k^2+6k+3). 
\end{equation}
But we know that $\mu_1^+(t_0^2)=2$. From (\ref{Y3}), we obtain 
\begin{eqnarray*}
\mu_2^+(t_0^2) \ \leq \ 2\ (2k^2+6k+3)^{\ },  \\ 
\mu_3^+(t_0^2) \ \leq \ 2\ (2k^2+6k+3)^2, \\
\cdots \cdots \cdots\ \ \ \ \ \ \ \ \ \ \ \ 
\end{eqnarray*}
At the end of $m$-th trip, we have
\begin{equation}\label{Y4} 
\mu_m^+(t_0^2)\leq 2\ (2k^2+6k+3)^{m-1}. 
\end{equation}
We have thus almost proved the following lemma. 
\begin{lem}
Let $p=0$ and let $f\in \mathscr{F}_0$ be determined by a primitive filling closed geodesic $\gamma$ as shown in Figure $7$, and suppose that $c$ is a boundary of a twice punctured disk enclosing $x_1$ and $x_{n-1}$ (shown as in Figure $7$ also). Then for any integer $m\geq	 0$, 
\begin{equation}\label{IN}
i(f^{1+m}(c),c)\leq 4\ \sum_{j=0}^m \left(2k^2+6k+3\right)^j,
\end{equation}
where $k=i_{\gamma}=n-3$. 
\end{lem} 
\begin{proof}
We have seen from  (\ref{GG})  that (\ref{IN}) is true when $m=0$. Suppose that (\ref{IN}) is true for an integer $m-1$ for $m\geq 1$.
By (\ref{Y4}), at the start of the $(m+1)$-th trip, $x_t$ carries at most $2\ (2k^2+6k+3)^{m-1}$ strands. 
From (\ref{Y3}) and (\ref{Y4}) we see that when the current loop is completed, $x_t$ carries at most 
$$
2\ (2k^2+6k+3)^{m-1}(2k^2+6k+3)=2\ (k^2+6k+3)^m
$$ 
strands. As each strand contributes two points of intersection between $f^{1+m}(c)$ and $c$, it follows that the total number of points of intersection between $f^{1+m}(c)$ and $c$ satisfies
\begin{equation}\label{Y5}
i(f^{1+m}(c),c)\leq i(f^m(c),c)+\mu_{m+1}^+(t_0^2).
\end{equation} 
But from hypothesis we know that 
$$
i(f^{m}(c),c)\leq  4\ \sum_{j=0}^{m-1} \left(2k^2+6k+3\right)^j
$$
and 
$$
\mu_{m+1}^+(t_0^2)\leq 4\ (2k^2+6k+3)^m.
$$
Hence from (\ref{Y5}) we obtain 
$$
i(f^{1+m}(c),c)\leq 4\ \sum_{j=0}^{m} \left(2k^2+6k+3\right)^j.
$$
The assertion then follows. 
\end{proof}
\noindent {\em Proof of Theorem $\ref{T5}$: } 
Write $\beta=2k^2+6k+3$ and $b=f(c)$. By the same argument as in (3.6) and (3.7), we obtain 
$$
\frac{i(f^m(b),c)}{\lambda(f)^m}\leq \frac{2}{\beta-1}\left\{ \beta \left( \frac{\beta}{\lambda(f)}  \right)^m - \frac{1}{\lambda(f)^m}      \right\}
$$
By Theorem \ref{TT} again, we conclude that 
\begin{equation}\label{Y6}
\lambda(f)\leq \beta=2k^2+6k+3=2i_{\gamma}^2+6i_{\gamma}+3.
\end{equation} 

Note that $f_n$ is defined on $S_n$ and that the curve $\gamma$ is on the sphere $S_{n-1}$ with $n-1$ punctures. From (\ref{Y6}) (with $k$ being replaced by $i_{\gamma}=n-3$), we obtain (\ref{Y7}). This completes the proof of Theorem 4.1. \ \ \ \ \ \ \ \ \ \ \ \ \ \ \ \ \ \ \ \ \ \ \ \ \ \ \ \ \ \ \ \ \ \ \ \ \ \ \ \ \ \ \ \ \ \ \ \ \ \ \ \ \ \ \ \ \ \ \ \ \ \ \ \  $\Box$

\medskip

\noindent {\em Proof of Theorem $\ref{T2}$ (in upper bound case): } By Theorem 4.1, we can find an element $f\in \mathscr{F}_0$ such that $\lambda(f)\leq 2n^2-6n+3$. This implies that 
$$
L(\mathscr{F}_0)\leq \mbox{log}\left(2n^2-6n+3\right).
$$ 
Since $\mathscr{F}_0\subset \mathscr{F}$, $L(\mathscr{F})\leq L(\mathscr{F}_0)$. We conclude  that $L(\mathscr{F})\leq \mbox{log}\left( 2n^2-6n+3 \right)$. \ \ \ \ \ \ \ \ \ \ \ \ \ \ \ \ \ \ \ \ \ \ \ \ \ \ \ \ \ \ \ \ \ \  \ $\Box$

\section{Proof of Theorem \ref{T3} and Theorem \ref{T4}}
\setcounter{equation}{0}

\noindent {\em Proof of Theorem $\ref{T3}$: } Let $S_n$ be as in (\ref{SPHERE}) and suppose that $f_{n}$ is a pseudo-Anosov map on $S_n$ that has a minimum dilatation $\lambda(f_{n})$. Further suppose that $f_n$ has the property that $f_{n}(x_i)=x_i$ for some $x_i\in \{x_1,\ldots,x_n\}$, and that $f_{n}$ defines a trivial mapping class on $S\cup \{x_i\}$. This means that $f_{n}(x_j)=x_j$ for all $j=1,2,\ldots, n$.  For simplicity we write $\tilde{S}_i=S\cup \{x_i\}$. Observe that $\tilde{S}_i$ is of type $(0,n-1)$. 

From the discussion in Section 2, we see that $f_{n}$ is defined by a filling closed geodesic $\gamma=\gamma_i$ on $\tilde{S}_i$, and conversely, given a filling closed geodesic $\gamma$ on $\tilde{S}_i$, once a base point $x$ on $\gamma$ is chosen, a pseudo-Anosov map $f_{n}$ can be obtained by pushing the point $x$ along the curve $\gamma$ until it returns to its original location. See also \cite{D} for more details. 
   
By assumption and Theorem 1.1, we have
\begin{equation}\label{KK}
\lambda(f_{n})=L(\mathscr{F})\geq \mbox{log}\left(2n-5\right).
\end{equation}
Now we invoke a result of Hironaka-Kin \cite{H-K}, which states that 
\begin{equation}\label{EEEE}
L(\mbox{Mod}_{S_n}) < \frac{2\ \mbox{log}\left( 2+\sqrt{3}\right)}{n-3}.
\end{equation}
By assumption we have $\lambda(f_n)=L(\mbox{Mod}_{S_n})$. Therefore, (\ref{EEEE}) and (\ref{KK}) combine to yield
$$
7+4\sqrt{3}>\left(2n-5  \right)^{n-3}.
$$
But this inequality holds only for $n\leq 4$. In the case when $n=4$, the example in the introduction shows that $\sigma_2\sigma_1^{-1}$ defines a pseudo-Anosov mapping class on $\mathbf{S}^2\backslash \{x_1,x_2,x_3,x_4\}$ with the minimum dilatation and has the desired property. 

To prove the second statement of Theorem \ref{T3}, we suppose that $f_n^q$ (for some $q\leq n$) is isotopic to the identity on $\tilde{S}_i$ for an $x_i\in \{x_1,\ldots, x_n\}$. Note that $\tilde{S}_i$ is of type $(0,n-1)$. By Theorem \ref{T2} again, we have 
$$
\log \lambda(f_n^q)\geq \mbox{log}\left(2n-5\right).
$$
But by the assumption and the result in \cite{H-K},
\begin{equation}\label{BB}
\log \lambda(f_n)=L(\mbox{Mod}_{S_n})< \frac{2\ \mbox{log}\left(2+\sqrt{3}\right)}{n-3}. 
\end{equation}
We thus obtain 
$$
\mbox{log}\left(2n-5\right)\leq q\ \log \lambda(f_n)<\frac{2q\ \mbox{log}\left(2+\sqrt{3}\right)}{n-3}\leq \frac{2n\ \mbox{log}\left(2+\sqrt{3}\right)}{n-3}.
$$
This is equivalent to 
\begin{equation}\label{BB0}
\left( 2n-5  \right)^{n-3}< \left( 7+4\sqrt{3}  \right)^n.
\end{equation}
But (\ref{BB0}) fails when $n\geq 16$. This completes the proof of Theorem \ref{T3}. \ \ \ \ \ \ \ \ \ \ \ \ \ \ \ $\Box$

\bigskip

\noindent {\em Proof of Theorem $1.2$: } Suppose such an $f_n\in \mathscr{F}$ exists. Then by Theorem 1.1, 
\begin{equation}\label{V1}
L(\mathscr{F})\leq \mbox{log}\left(2n^2-6n+3\right). 
\end{equation}
On the other hand, if $f_n$ is represented by a product of two Dehn twists along two simple closed geodesics $a$ and $b$, then by  \cite{CZ3}, either both $a$ and $b$ are nontrivial on $S_{n-1}$, or both $a$ and $b$ are trivial on $S_{n-1}$. If both $a$ and $b$ are trivial on $S_{n-1}$, then Lemma 4.1 and Lemma 5.1 of \cite{CZ3} implies that 
\begin{equation}\label{U1}
\lambda(f_n)>h_1(4n-10),
\end{equation}
where $h_1(z)=\tfrac{1}{2}\left( z^2-2+z\sqrt{z^2-4}\right)$. If both $a$ and $b$ are nontrivial on $S_{n-1}$, then by Lemma 4.3 and Lemma 5.2 of \cite{CZ3}, 
\begin{equation}\label{U2}
\lambda(f_n)>h(2n-6),
\end{equation}
where $h(z)=\tfrac{1}{2}\left( z^2+2+z\sqrt{z^2+4}\right)$. 

Our first claim is that $h_1(4n-10)>h(2n-6)$. To see this, we first note that for any $z\geq 4$, 
$$
12z^4-36z^2+25>0.
$$
Elementary calculations then show that
\begin{equation}\label{W1}
4\sqrt{1-\frac{1}{z^2}}>1+\sqrt{1+\frac{4}{z^2}}. 
\end{equation}
By setting $z=2n-6$ for $n\geq 4$, it follows from (\ref{W1}) that 
\begin{equation}\label{W2}
4(2n-6)^2+4(2n-6)\sqrt{(2n-6)^2-1}>4+(2n-6)^2+(2n-6)\sqrt{(2n-6)^2+4}.
\end{equation}
Since $2n-6<2n-5$, inequality (\ref{W2}) yields that 
$$
(4n-10)^2+(4n-10)\sqrt{(4n-10)^2-4}>4+(2n-6)^2+(2n-6)\sqrt{(2n-6)^2+4}.
$$
The above inequality implies $h_1(4n-10)>h(2n-6)$. 

Now from (\ref{V1}), (\ref{U1}), and (\ref{U2}), we obtain
$$
2n^2-6n+3>\tfrac{1}{2}\left( 2+(2n-6)^2+(2n-6)\sqrt{4+(2n-6)^2}\right), 
$$
which is equivalent to 
\begin{equation}\label{Z1}
\frac{n^2}{(n-3)^2}-\frac{3n}{(n-3)^2}+\frac{1}{(n-3)^2}>1+\sqrt{\frac{1}{(n-3)^2}+1}.
\end{equation}
But it is easy to check that when $n\geq 7$, 
$$
n^2-3n+1<2(n-3)^2.
$$ 
It turns out that for $n\geq 7$,
\begin{equation}\label{Z2}
\frac{n^2}{(n-3)^2}-\frac{3n}{(n-3)^2}+\frac{1}{(n-3)^2}<2.
\end{equation}
Clearly, (\ref{Z2}) contradicts (\ref{Z1}). This completes the proof of Theorem 1.4. \ \ \ \ \ \ $\Box$


\section{Appendix}
\setcounter{equation}{0}

In this section, we outline the proof of Theorem \ref{K} by different methods. Write $f=g^*$, where the axis $c_g$ of $g$ projects to a filling closed geodesic $\gamma$ on $S\cup \{x\}$ under the universal covering map $\varrho:\mathbf{H}\rightarrow S\cup \{x\}$.  Suppose that for some integer $m$, we have 
\begin{equation}\label{JH0}
f^m(c)=c
\end{equation}
for a simple closed geodesic on $S$. If $c$ is trivial on $\tilde{S}$, then by Lemma 5.1 of \cite{CZ0}, $g$ would fix a parabolic fixed point of $G$. This is impossible since $g\in G$ is hyperbolic and $G$ is discrete. Assume that $c$ projects to a nontrivial geodesic $\tilde{c}$ on $\tilde{S}$. Then by Lemma 3.2 of \cite{CZ1}, there is a lift $\tau_c$ of the Dehn twist $t_{\tilde{c}}$ such that 
\begin{equation}\label{JH}
\varphi\circ \tau_c\circ \varphi^{-1}=t_c.
\end{equation}
 Now as a quasiconformal homeomorphism of $\mathbf{H}$, $\tau_c$ determines a collection $\mathscr{U}_c$ of maximal half planes in $\mathbf{H}$ so that $\mathbf{H}\backslash \mathscr{U}_c$  is a simply connected, convex region $\mathscr{R}$ with geodesic boundaries so that $\tau_c|_{\mathscr{R}}=\mbox{id}$. From (\ref{JH0}), (\ref{JH}) and Lemma 5.1 of \cite{CZ2}, we conclude that $g$ and hence $g^m$ must send any maximal element of $\mathscr{U}_c$ to a maximal element of $\mathscr{U}_c$.
 
On the other hand, since $\gamma=\varrho(c_g)$ is filling, it intersects $\tilde{c}$. Thus the geodesic $c_g\subset \mathbf{H}$ must intersect some geodesics in $\{\varrho^{-1}(\tilde{c})\}$. Note that all boundaries of elements of $\mathscr{U}_c$ are contained in $\{\varrho^{-1}(\tilde{c})\}$. There remain two cases to consider: (1) $c_g$ is contained in a maximal element $\Delta$ of $\mathscr{U}_c$, or (2) $c_g$ intersects the geodesic $\partial \Delta$ for some $\Delta\in \mathscr{U}_c$. In both cases, it is easy to see that $g^m(\Delta)$ is no longer a maximal element of $\mathscr{U}_c$. This contradiction proves that $\mathscr{F}_{\gamma}\subset \mathscr{F}$. 
 
To prove $\mathscr{F}\subset \bigcup \mathscr{F}_{\gamma}$, we choose an element $f\in \mathscr{F}$. The isotopy $I(\cdot, t)$ between $f$ and the identity gives rise to a Jordan curve $\gamma'$. Let $\gamma$ be the geodesic homotopic to $\gamma'$. If $\gamma$ is not filling, there is a curve $\tilde{c}$ disjoint from $\gamma$. Then one can obtain a geodesic $c\subset S$, homotopic to $\tilde{c}$ if $c$ is viewed as a curve on $S\cup \{x\}$, so that $f(c)=c$. This is impossible. Hence $\gamma$ must fill $S\cup \{x\}$. Finally, by combining the work of Bers \cite{Bers1} and the work of Birman \cite{Bir}, we conclude that  $f\in \mathscr{F}_{\gamma}$. This shows that $\mathscr{F}\subset \bigcup \mathscr{F}_{\gamma}$.

\end{document}